\documentclass[12pt,a4paper]{article}

\usepackage[OT1]{fontenc}
\usepackage[utf8]{inputenc}

\usepackage[english]{babel}

\usepackage{amsfonts}
\usepackage{amsmath}
\usepackage{amsthm}
\usepackage{amssymb}
\usepackage[all]{xy}
\newtheorem{teo}{Theorem}[section]
\newtheorem{prop}[teo]{Proposition}
\newtheorem{lem}[teo]{Lemma}
\newtheorem{coro}[teo]{Corollary}
\newtheorem{defi}[teo]{Definition}
\newtheorem{rem}[teo]{Remark}
\newtheorem{ejem}[teo]{Example}

\def\g{\mathfrak g}

\def\p{\mathfrak p}
\def\s{\mathfrak s}
\def\u{\mathfrak u}
\def\h{\mathcal H}
\def\b{\mathcal B}

\def\bh{{\cal B}({\cal H})}
\def\o{{\cal O}}

\vspace{-11mm}

\begin{document}

\title{\vspace*{0cm}Decompositions and complexifications of homogeneous spaces}
\date{}
\author{Martin Miglioli\footnote{\textit{e-mail: martin.miglioli@gmail.com}. Supported by Agencia Nacional de Promoción Científica y Tecnológica, and Instituto Argentino de Matemática \textit{Alberto P. Calderón}, CONICET. Saavedra 15, Piso 3, (1083) Buenos Aires, Argentina. Tel/Fax +54 011 4954 6781.}}

\maketitle

\abstract{\footnotesize\noindent In this paper an extended CPR decomposition theorem for Finsler symmetric spaces of semi-negative curvature in the context of reductive structures is proven. This decomposition theorem is applied to give a geometric description of the complexification of some infinite dimensional homogeneous spaces.

\medskip

\noindent{\scriptsize{{\bf Keywords:} Banach-Lie group, coadjoint orbit, complexification, Corach-Porta-Recht decomposition, Finsler structure, flag manifold, homogeneous space, operator decomposition, reductive structure, Stiefel manifold.}}}


\setlength{\parindent}{0cm} 

\section{Introduction}

In recent years, the geometrical study of operator algebras and their homogeneous spaces has become a central topic in the study of infinite dimensional geometry. It is a source of examples and counterexamples, and the operator algebra techniques (Banach algebras and $C^*$ algebras, with their distinguished tools) are being used for obtaining results on abstracts infinite dimensional manifolds by studying their groups of automorphism, isometries, and their associated fiber bundles and $G$-bundles. See the recent book \cite{beltita} by D. Beltita for a full account of these objects and a comprehensive list of references.

In particular, what we are interested in here, is the extension of certain results on the geometric description of complexifications of homogeneous spaces of Banach-Lie groups studied by Beltita and Gal\'e in \cite{beltitagale} and also the decompositions of the acting groups by means of a series of chained reductive structures. 

In Section 2 the reader can find the basic facts about Finsler symmetric spaces; these are spaces of the form $G/U$ endowed with a Finsler structure, where $G$ is a Banach-Lie group and $U$ is the fixed point set of an involution $\sigma$ on $G$.  A criteria that ensures that the spaces $G/U$ have semi-negative curvature is recalled from the work of Neeb \cite{neeb}.

In Section 3 we recall the definition of reductive structures, which can be interpreted as connection forms $E$ on homogeneous spaces of the form $G_A/G_B$. Examples in the context of operator algebras are given: conditional expectations, their restrictions to Schatten ideals and projections to corners of operator algebras. The Corach-Porta-Recht splitting theorem by Conde and Larotonda \cite{condelarotonda} is used to prove an extended CPR-splitting theorem in the context of several reductive structures. 

In Section 4 the CPR splitting theorem is used to give a geometric description of homogeneous spaces of the form $G_A/G_B$ as associated principal bundles over $U_A/U_B$. Under additional hypothesis about the holomorphic character of $G_A$ and the involution $\sigma$ on $G_A$ it is possible to interpret $G_A/G_B$ as the complexification of $U_A/U_B$. Under these additional assumptions $G_A/G_B$ is identified with the tangent bundle of $U_A/U_B$ and it is shown that this identification has nice functorial properties related to the connection form $E$. Finally, we use the three examples of connection forms introduced in Section 3, to give a geometrical description of the complexifications of flag manifolds, coadjoint orbits in Schatten ideals and Stiefel manifolds respectively.

\section{Finsler symmetric spaces}

A connected Banach-Lie group $G$ with an involutive automorphism $\sigma$ is called a \textit{symmetric Lie group}. Let $\g$ be the Banach-Lie algebra of $G$, and let $U=\{g\in G:\sigma(g)=g\}$ be the subgroup of fixed points of $\sigma$. Then the Banach-Lie algebra $\mathfrak{u}$ of $U$ is a closed and complemented subspace of $\g$; a complement is given by the closed subspace
$$\p=\{X\in \g : \sigma_{*1}X=-X\},$$
where for a smooth map between manifolds $f:X\to Y$ we use the notation $f_*:T(X)\to T(Y)$ for the tangent map and $f_{*x}:T_x(X)\to T_{f(x)}(Y)$ for the tangent map at $x\in X$.

The Lie algebra $\u$ is the eigenspace of $\sigma_{*1}$ corresponding to the eigenvalue $+1$ and $\p$ is the eigenspace corresponding to the eigenvalue $-1$. Since $\u$ is complemented $U$ is a Banach-Lie subgroup of $G$, and the quotient space $M=G/U$ has a Banach manifold structure. We denote by $q:G\to M$, $g\mapsto gU$ the quotient map which is a submersion, and by $Exp:\g\to G$ the exponential map of $G$. We use the notation $e^X:=Exp(X)$ for $X\in \g$. 

We also define $G^+ := \{g\sigma(g)^{-1} : g \in G\}$ which is a submanifold of $G$ and note that there is a differential isomorphism $\phi: G/U\to G^+$, $gU\mapsto g\sigma(g)^{-1}$. See Section 5 in Chapter XIII of \cite{lang}. We use the notation $\sigma(g)^{-1}:=g^*$ for $g\in G$.

Let $L_g$ and $R_g$ stand for the left and right translation diffeomorphisms on $G$. For $h\in G$, let $\mu_h:M\to M$, $\mu_h(q(g))=q(hg)=q(L_hg)$. Then
$$(\mu_h)_{*q(g)} q_{*g}=q_{*hg}(L_{h})_{*g}.$$ 
The map $q_{*1}:\p\to T_oM$ is an isomorphism so that a generic vector in $T_{q(g)}M$ will be denoted by $(\mu_g)_{*o}q_{*1} X$ with $X\in \p$.
We use $I_h$ to denote the interior automorphisms of $G$ given by $I_h(g)=hgh^{-1}$, and $Ad_h$ to denote the differential $(I_h)_{*1}$, which is an element of ${\cal B}(\g)$, the bounded linear maps that act on $\g$. We note that $\sigma(I_ue^{tX})=I_u e^{-tX}$ for every $X\in \p$ and $u\in U$, so that $\sigma_{*1}Ad_u X=-Ad_u X$ and $\p$ is $Ad_U$-invariant.
Since $\sigma$ is a group automorphism, $\sigma_{*1}$ is an automorphism of Lie algebras and the following inclusions hold:
$$[\u,\u]\subseteq \u,\quad [\u,\p]\subseteq \p,\quad [\p,\p]\subseteq \u.$$
In particular, $\p$ is $ad_{\u}$-invariant.

A way of giving $M$ the structure of a Finsler manifold is establishing the following norm on the tangent space $T_{q(g)}(M)$ for each $g\in G$
$$\|(\mu_g)_{*o}q_{*1}X\|_{q(g)}:=\|X\|_{\p}$$
where $\|\cdot\|_{\p}$ is any $Ad_U$-invariant norm $\p$ compatible with any norm of $T_{o}(M)$ given by a local chart.
To make the dependence of $M$ with its underlying Banach-Lie group, involution and Finsler structure clear we shall write $M=G/U=Sym(G,\sigma,\|\cdot\|_{\p})$ and we shall call $M$ a \textit{Finsler symmetric space}.

We say that $M=G/U$ has semi-negative curvature if for all $p\in M$ the operator between Banach spaces $(exp_p)_{*x}:T_p(M) \simeq T_x(T_p(M))\to T_{exp_p(x)}(M)$ is expansive and surjective.

What follows is a criteria for semi-negative curvature for Finsler symmetric spaces due to K.-H. Neeb, \cite[Prop. 3.15 and Th. 2.2]{neeb}:

\begin{teo}\label{csn}
Let $M=G/U=Sym(G,\sigma,\|\cdot\|_{\p})$ be a Finsler symmetric space. Then the following conditions are equivalent:
\begin{enumerate}
\item $M$ has semi-negative curvature.
\item The operator $-(ad_X)^2|_{\p}$ is dissipative for all $X\in \p$.
\item The operator $1+(ad_X)^2|_{\p}$ is expansive and invertible for all $X\in \p$.
\item The operator $X\in \p$,  $\frac{\sinh ad_X}{ad_X}|_{\p}$ is expansive and invertible for all $X\in \p$.
\end{enumerate}
\end{teo}

\begin{ejem}\label{calg}
If $A$ is a unital $C^*$-algebra, $G$ is the group of invertible elements of $A$ endowed with the manifold structure given by the norm and $\sigma:G\to G$, $g\mapsto (g^{-1})^*$, then $U=\{g\in G:\sigma(g)=g\}$ is the group of unitary operators of $A$. In this case $\p=A_s$ the set of self-adjoint elements of $A$ and the uniform norm on $A_s$ which we denote by $\|\cdot\|$ is $Ad_U$-invariant because it is unitarily invariant. We can identify the manifold $G/U$ with the manifold of positive invertible elements $G^+$. It was proven in \cite{cpr} that the manifold $M=G/U=Sym(G,\sigma,\|\cdot\|)$ has semi-negative curvature.
\end{ejem}

\begin{ejem}\label{ideal}
Let $A=\bh$ stand for the set of bounded linear operators on a separable complex Hilbert space ${\cal H}$, with the uniform norm denoted by $\|\cdot\|$. Let $A_p$ be the ideal of $p$-Schatten operators with $p$-norm $\|\cdot\|_p$. Let $G_p$ stand for the group of invertible operators in the unitized ideal, that is $G_p=\{g\in A^{\times}: \;g-1\in A_p\}$, then $G_p$ is a Banach-Lie group (one of the so-called classical Banach-Lie groups \cite{harpe}), and $A_p$ identifies with its Banach-Lie algebra. Consider the involutive automorphism $\sigma:G_p \to G_p$ given by $g\mapsto (g^{*})^{-1}$. Let $U_p\subseteq G_p$ stand for the unitary subgroup of fixed points of $\sigma$. In this case $\p$ is the set of self-adjoint operators in $A_p$ and the norm $\|\cdot\|_p$ on $\p$ is $Ad_{U_p}$-invariant. We can identify the manifold $G_p/U_p$ with the manifold of positive invertible operators in $G_p$. It was proven in Section 5 of \cite{condelarotonda} that the manifold $M_p=G_p/U_P=Sym(G_p,\sigma,\|\cdot\|_p)$ is simply connected and has semi-negative curvature.
\end{ejem}

\section{Splitting of Finsler symmetric spaces}

We recall some facts about the fundamental group of $M$ and polar decompositions \cite[Th. 3.14 and Th. 5.1]{neeb}

\begin{teo}\label{polar}
Let $M=G/U=(G,\sigma,,\|\cdot\|_{\p})$ be a Finsler symmetric space of semi-negative curvature, then
\begin{enumerate}
\item The exponential map $q\circ Exp:\p\to M$ is a covering of Banach manifolds and
$$\Gamma=\{X\in \p: q(e^X)=q(1)\}$$
is a discrete and additive subgroup of $\p \cap Z(\g)$, with $\Gamma\simeq \pi_1(M)$ and $M\simeq \p/\Gamma$. $Z(\g)$ denotes the center of the Banach-Lie algebra $\g$. If $X,Y\in\p$ and $q(e^X)=q(e^Y)$, then $X-Y\in\Gamma$.
\item The polar map 
$$m:\p\times U \to G, \quad (X,u)\mapsto  e^X u$$ 
is a surjective covering whose fibers are given by the sets $\{ (X-Z,e^Z u):\, \, Z \in \Gamma \, \}$, $u\in U$, $X\in \p$. If $M$ is simply connected the map $m$ is a diffeomorphism.
\end{enumerate}
\end{teo}

In the context of $C^*$-algebras (Example \ref{calg}), since $G/U$ is simply connected and has semi-negative curvature we get the usual polar decomposition of invertible elements as a product of a positive invertible element and a unitary.

\begin{coro}\label{exppos}
In the context of the previous theorem $G_A^+ = e^{\p}$. Note that given $h\in G_A^+$ there is a $g\in G_A$ such that $h=g\sigma(g)^{-1}$. Using the polar decomposition in $G_A$ there are $X\in \p$ and $u\in U$ such that $g=e^Xu$. Then $h=e^Xu\sigma(e^Xu)^{-1}=e^Xuu^{-1}e^X=e^{2X}\in e^{\p}$. We note also that $e^X=e^{\frac{1}{2}X}\sigma(e^{\frac{1}{2}X})^{-1}\in G_A^+$ for every $X\in \p$.
\end{coro}

The following decomposition theorem in the context of Finsler symmetric spaces of semi-negative curvature was proven by Conde and Larotonda in \cite{condelarotonda}.

\begin{teo}Corach-Porta-Recht decomposition (CPR)\label{cpr}

Let $M=G/U=(G,\sigma,,\|\cdot\|_{\p})$ be a simply connected Finsler symmetric space of semi-negative curvature.
Let $p\in {\cal B}(\p)$ be an idempotent, $p^2=p$. Let $\s:=Ran(p)$, $\s':=Ran(1-p)=Ker(p)$, so that $\p=\s\oplus\s'$. If $ad^2_{\s}(\s)\subseteq \s$, $ad^2_{\s}(\s')\subseteq \s'$ and $\|p\|=1$, then the maps
$$\Phi: U \times \s'\times \s \to G, \qquad (u,X,Y)\mapsto ue^Xe^Y$$
$$\Psi: \s'\times \s \to G^+, \qquad (X,Y) \mapsto e^Ye^{2X}e^Y$$
are diffeomorphisms.
\end{teo}

The following two definitions are from Beltita and Gal\'e \cite{beltitagale2}.

\begin{defi} \label{estred}
A \textit{\textbf{reductive structure}} is a triple $(G_A,G_B;E)$ where $G_A$ is a real or complex connected Banach-Lie group with Banach-Lie algebra $\mathfrak{g}_A$, $G_B$ is a connected Banach-Lie subgroup of $G_A$ with Banach-Lie algebra $\mathfrak{g}_B$, and $E: \mathfrak{g}_A \to \mathfrak{g}_A$ is a $\mathbb{R}$-linear continuous transformation which satisfies the following properties: $E \circ E = E$; $RanE = \mathfrak{g}_B$, and for every $g \in G_B$ the diagram
$$
\xy
\xymatrix{
\mathfrak{g}_A \ar[d]_{Ad_g} \ar[r]^{E} & \mathfrak{g}_B \ar[d]^{Ad_g}
\\
\mathfrak{g}_A \ar[r]^{E} & \mathfrak{g}_B
}\endxy
$$
commutes.
\end{defi}

\begin{defi}
A \textit{\textbf{morphism of reductive structures}} from $(G_A,G_B;E)$ to $(\tilde{G_A},\tilde{G_B};\tilde{E})$ is a homomorphism of Banach-Lie groups $\alpha: G_A \to \tilde{G_A}$ such that $\alpha (G_B) \subseteq \tilde{G_B}$ and such that the diagram
$$
\xy
\xymatrix{
\mathfrak{g}_A \ar[d]_{\alpha_{*1}} \ar[r]^{E} & \mathfrak{g}_B \ar[d]^{\alpha_{*1}}
\\
\tilde{\mathfrak{g}_A} \ar[r]^{\tilde{E}} & \tilde{\mathfrak{g}_B}
}\endxy
$$
commutes. 

For example, a family of automorphisms of any reductive structure $(G_A,G_B;E)$ is given by $\alpha_g : x \mapsto gxg^{-1}$, $G_A \to G_A$, $(g \in G_B)$.
\end{defi}

Now we introduce involutions in reductive structures:

\begin{defi}
If  $(G_A,G_B;E)$ is a reductive structure and $\sigma$ is an involutive morphism of reductive structures we call  $(G_A,G_B;E,\sigma)$ a \textit{\textbf{reductive structure with involution}}. If $(G_A,G_B;E,\sigma)$ and $(\tilde{G_A},\tilde{G_B};\tilde{E},\tilde{\sigma})$ are reductive structures with involution and $\alpha$ is a morphism of reductive structures from $(G_A,G_B;E)$ to $(\tilde{G_A},\tilde{G_B};\tilde{E})$ such that $\alpha \circ \sigma=\tilde{\sigma} \circ \alpha$ then we call $\alpha$ a \textit{\textbf{morphism of reductive structures with involution}} from $(G_A,G_B;E,\sigma)$ to $(\tilde{G_A},\tilde{G_B};\tilde{E},\tilde{\sigma})$.
\end{defi}

\begin{ejem}\label{condexp} Conditional expectations in $C^*$-algebras

Let $A$ and $B$ be two unital $C^*$-algebras, such that $B$ is a subalgebra of $A$ and let $E:A\to B$ be a conditional expectation. This means that $E$ is a linear projection on $A$ with $RanE=B$, $E(1_A)=1_B(=1_A)$ and norm $1$. By Tomiyama's theorem \cite{tomi} the following holds
$$E(b_1ab_2)=b_1E(a)b_2 \qquad \textrm{for all $a\in A;\quad b_1,b_2\in B$}$$
$$E(a^*)=E(a)^* \qquad \textrm{for all $a\in A$}.$$
Let $G_{\Lambda}$ for ${\Lambda}\in \{A,B\}$ be the Banach-Lie group of invertible operators in ${\Lambda}$ endowed with the topology given by the uniform norm. Then the Banach-Lie algebra of $G_{\Lambda}$ is $\g_{\Lambda}={\Lambda}$. Since in this case we have $Ad_g(a)=gag^{-1}$ for each $g\in G_A$ and $a\in A$, the expectation $E$ satisfies the conditions of Def. \ref{estred}, so that $(G_A,G_B;E)$ is a reductive structure. In fact, this is a classical example that was the motivation of that definition in the paper \cite{beltitagale2}.

If $(G_A,G_B;E)$ is a reductive structure that is derived from an inclusion of $C^*$-algebras and a conditional expectation as above then $\sigma:G_A\to G_A$, $a\mapsto (a^{-1})^*$ defines an involutive morphism of reductive structures since $\sigma_{*1}:A\to A$, $a\mapsto -a^*$ and 
$$
E(\sigma_{*1}(a))=E(-a^*)=-E(a)^*=\sigma_{*1}(E(a)),
$$
therefore $(G_A,G_B;E,\sigma)$ is a reductive structure with involution.

If for two triples $(A,B;E)$, $(\tilde{A},\tilde{B};\tilde{E})$ there is a bounded homomorphism $\phi:A\to \tilde{A}$ which satisfies $\phi \circ E=\tilde{E}\circ \phi$ then $\alpha := \phi|_{G_A}$ defines a morphism of reductive structures with involution from $(G_A,G_B;E,\sigma)$ to $(\tilde{G_A},\tilde{G_B};\tilde{E},\tilde{\sigma})$.
\end{ejem}

\begin{ejem}\label{idealcond}
We use the notation of Example \ref{ideal}. Let $B\subseteq A=\b(\h)$ be a $C^*$-subalgebra, and let $E:A\to B$ be a conditional expectation with range $A$ such that $E$ sends trace-class operators to trace-class operators and $E$ is compatible with the trace, that is $Tr(E(x))=Tr(x)$ for any trace-class operator $x\in A$. Let $p\geq 1$, $B_p=B\cap A_p$, 
$$
G_{A,p}=\{g\in A^{\times}:g-1\in A_p\}\; \textrm{ and }\; G_{B,p}=\{g\in A^{\times}:g-1\in B_p\}.
$$
Then $\g_{A,p}:=A_p$ and $\g_{B,p}=B_p$ are the Banach-Lie algebras of $G_{A,p}$ and $G_{B,p}$ respectively. It was proven in Section 5 of \cite{condelarotonda} that $E_p=E\vrule_{A_p}:A_p \to B_p$ and that $\|E_p\|=1$. It easy to see that $(G_{A,p},G_{B,p}; E_p,\sigma)$ is a reductive structure with involution.
\end{ejem}

\begin{ejem}\label{corners} Corners

Let $\h$ be a Hilbert space, $n\geq1$ and $p_i$, $i=1,\dots ,n+1$ be pairwise orthogonal non-zero projections with range $\h_i$ and $\sum_{i=1}^{n-1}p_i =1$. Let $G_A$ be the group of invertible elements of $\b(\h)$ and let  
$$
G_B=
\left\{ \begin{pmatrix}
g_1    & 0      & \hdots & 0    &   0 \\
0      & g_2    & \hdots &  0   &   0 \\
\vdots & \vdots & \ddots & \vdots  &  \vdots \\
0      &  0     &  \hdots   &  g_n  & 0   \\
 0      &   0     &     \hdots      &   0    &  1
\end{pmatrix} :g_i \mbox{ invertible in } \b(\h_i) \mbox{ for } i=1,\dots ,n \right\};
$$
where we write operators in $\b(\h)=\b(\h_1\oplus \ldots \oplus \h_{n+1})$ as $(n+1)\times (n+1)$ matrices with the corresponding operator entries.  

In this case $\g_A=\b(\h)$ and 
$$
\g_B=
\left\{ \begin{pmatrix}
X_1    & 0      & \hdots & 0    &   0 \\
0      & X_2    & \hdots &  0   &   0 \\
\vdots & \vdots & \ddots & \vdots  &  \vdots \\
0      &  0     &  \hdots   &  X_n  & 0   \\
 0      &   0     &     \hdots      &   0    &  0
\end{pmatrix} :X_i \mbox{ in } \b(\h_i) \mbox{ for } i=1,\dots ,n \right\}.
$$

If we consider the map $E:\g_A \to \g_B$, $X \mapsto \sum_{i=1}^{n}p_iXp_i$ and $\sigma=(\cdot)^{*-1}$ it is easily verified that $(G_A,G_B;E,\sigma)$ is a reductive structure with involution. Note that $\|E\|=1$. 
\end{ejem}

\begin{defi}
If $(G_A,\sigma)$ is a symmetric Banach-Lie group we say that a connected subgroup $G_B\subseteq G_A$ is \textit{\textbf{involutive}} if $\sigma(G_B)=G_B$.
\end{defi}

\begin{rem}
If $G_B\subseteq G_A$ is an involutive Banach-Lie subgroup with Banach-Lie algebra $\g_B\subseteq \g_A$ and $\g_A = \p \oplus \u$ is the eigenspace decomposition of $\sigma_{*1}$, we can write $\g_B=\p_B\oplus \u_B$, where $\p_B:=\p\cap \g_B$ and $\u_B:=\u\cap \g_B$.
\end{rem}

\begin{prop}\label{subgrupoinv}
Given a Finsler symmetric space 
$$
M_A=G_A/U_A=Sym(G_A,\sigma,\|\cdot\|_{\p})
$$
of semi-negative curvature, if $G_B$ is an involutive subgroup, then 
$$
M_B=G_B/U_B=Sym(G_B,\sigma|_{G_B},\|\cdot\|_{\p_B})
$$
is a Finsler symmetric space of semi-negative curvature.
Also, the inclusion $\Gamma_B \subseteq \Gamma_A \cap \p_B$ holds. In particular, if $M_A$ is simply connected then $M_B$ is also simply connected.
\end{prop}
\begin{proof}
We can restrict the $Ad_{U_A}$-invariant norm of $M_A=G_A/U_A$ to $\p_B$ to give $M_B=G_B/U_B$ a $Ad_{U_B}$-invariant norm. Since for each $X\in \p$ the operator $-(ad_X)^2|_{\p}$ is dissipative and $-(ad_X)^2|_{\p}(\p_B)\subseteq \p_B$ for all $X\in \p_B$, we conclude that the operator $-(ad_X)^2|_{\p_B}$ is dissipative for all $X\in \p_B$. Therefore $M_B=G_B/U_B=Sym(G_B,\sigma|_{G_B},\|\cdot\|_{\p_B})$ has semi-negative curvature.

If $X \in \Gamma_B$ then $q_B \circ Exp_B (X) = o_B$ so that $Exp_A (X) = Exp_B (X) \in U_B \subseteq U_A$ and $q_A \circ Exp_A = o_A$. We conclude that $\Gamma_B \subseteq \Gamma_A \cap \p_B$.
\end{proof}

\begin{rem}
If $(G_A,G_B;E)$ is a reductive structure, since $Ad_g\circ E=E\circ Ad_g$ for each $g\in G_B$ we see that $Ad_g(KerE)\subseteq KerE$ for every $g\in G_B$.
If $\sigma$ is an involutive automorphism of reductive structures and $\g_A = \u \oplus \p$ is the decomposition into eigenspaces of $\sigma_{*1}$ then $Ad_{U_A}(\p) \subseteq \p$ and $Ad_{U_A}(\u) \subseteq \u$, so that the actions $Ad:U_B \to {\cal B}(\p_E)$ and $Ad:U_B \to {\cal B}(\u_E)$ are well defined, where we denote $\p_E := KerE \cap \p$ and $\u_E := KerE \cap \u$.
\end{rem}

\begin{teo} Extended CPR splitting \label{descpolext}

If for $n \geq 2$ we have the following inclusions of connected Banach-Lie groups, the following maps between its Banach-Lie algebras
$$G_1\subseteq G_2\subseteq \dots  \subseteq G_n$$
$$\g_1 \xleftarrow{E_2} \g_2 \xleftarrow{E_3} \dots  \xleftarrow{E_n} \g_n$$
and a morphism $\sigma : G_n \to G_n$ such that:

\begin{itemize}
\item $(G_n,G_{n-1};E_n,\sigma)$,$(G_{n-1},G_{n-2};E_n,\sigma|_{G_{n-1}})$,\dots , $(G_2,G_1;E_2,\sigma|_{G_2})$ are reductive structures with involution.
\item $M_n=G_n/U_n=Sym(G_n,\sigma,\|\cdot\|)$ is a simply connected Finsler symmetric space of semi-negative curvature.
\item $\|{E_k}\vrule_{\,\p_k}\|=1$ for $k=2,\dots ,n$, where the norm is the norm of the previous item restricted to $\p_k:=\p \cap \g_k$.
\end{itemize}

Then the maps
$$\Phi_n: U_n \times \p_{E_n} \times \dots  \times \p_{E_2} \times \p_1 \to G_n$$
$$(u_n,X_n,\dots ,X_2,Y_1) \mapsto u_ne^{X_n}\dots e^{X_2}e^{Y_1}$$

$$\Psi_n: \p_{E_n} \times \dots  \times \p_{E_2} \times \p_1 \to G_n^+$$
$$(X_n,\dots ,X_2,Y_1) \mapsto e^{Y_1}e^{X_2}\dots e^{X_{n-1}}e^{2X_n}e^{X_{n-1}}\dots e^{X_2}e^{Y_1}$$
are diffeomorphisms, where $\p_{E_k}:=KerE_k\cap \p_k$ for $k=2,\dots ,n$.
\end{teo}

\begin{proof}
Note that Prop. \ref{subgrupoinv} implies that $M_k:=G_k/U_k$ are simply connected Finsler symmetric spaces of semi-negative curvature for  $k=2,\dots ,n$.
We prove the statement about the map $\Phi$ for the case $n = 2$ and then prove the statement for $n > 2$ by induction.

Since $E_2\circ\sigma_{*1}=\sigma_{*1}\circ E_2$, $E_2(\p_2)\subseteq \p_2$, we can consider $p:=E_2\vrule_{\,\p_2}: \p_2 \to \p_2$. We see that $\|p\|=1$ and $Ker(p)=Ran(1-p)= \p_{E_2}$. Also, since $E_2^2=E_2$ and $Ran(E_2)=\g_1$, $Ran(p)= \p_1$. The condition $ad^2_{\p_1}(\p_1)\subseteq \p_1$ of the statement of the CPR splitting \ref{cpr} is trivial. Also note that for every $g\in G_1$ and for every $X\in \g_2$, $Ad_g(E_2(X))=E_2(Ad_g(X))$. If $Y\in \g_1$ and we differentiate $Ad_{e^{tY}}(E_2(X))=E_2(Ad_{e^{tY}}(X))$ at $t=0$ we get $ad_Y(E_2(X))=E_2(ad_Y(X))$ and therefore $ad_{\g_1}(KerE_2) \subseteq KerE_2$. Since $ad^2_{\p_2}(\p_2)\subseteq \p_2$ we conclude that $ad^2_{\p_1}(\p_{E_2})\subseteq \p_{E_2}$.
The CPR splitting (Th. \ref{cpr}) implies the existence of a diffeomorphism
$$\Phi_2: U_2 \times \p_{E_2} \times \p_1 \to G_2$$
$$(u_2,X_2,Y_1) \mapsto u_2e^{X_2}e^{Y_1}.$$
Assume now that $n > 2$ and that the theorem is true for $k=n-1$ and $k=2$.
We prove that $\Phi_n$ is surjective. If $g_n \in G_n$ then the splitting theorem applied to the reductive structure $(G_n,G_{n-1};E_n)$ implies the existence of $u_n \in U_n$, $X_n \in \p_{E_n}$ and $Y_{n-1}$ such that $g_n = u_ne^{X_n}e^{Y_{n-1}}$. Since $e^{Y_{n-1}} \in G_{n-1}$ applying the splitting theorem in the case $k=n-1$ we get $u_{n-1} \in U_{n-1}$, $X_{n-1}\in \p_{E_{n-1}}$,\dots , $X_2 \in \p_{E_2}$ and $Y_1 \in \p_1$ such that $e^{Y_{n-1}} = u_{n-1}e^{X_{n-1}}\dots e^{X_2}e^{Y_1}$. Then
$$g_n= u_ne^{X_n}e^{Y_{n-1}}=u_ne^{X_n}u_{n-1}e^{X_{n-1}}\dots e^{X_2}e^{Y_1}=u_nu_{n-1}e^{Ad_{u_{n-1}^{-1}}X_n}e^{X_{n-1}}\dots e^{X_2}e^{Y_1}$$
is in the image of $\Phi_n$ because $Ad_{u_{n-1}^{-1}}X_n \in \p_{E_n}$.

We prove that $\Phi_n$ is injective. Assume that
$$u_ne^{X_n}e^{X_{n-1}}\dots e^{X_2}e^{Y_1}=u_n'e^{X_n'}e^{X_{n-1}'}\dots e^{X_2'}e^{Y_1'}.$$
Since $e^{X_{n-1}}\dots e^{X_2}e^{Y_1}\in G_{n-1}$ there are $u_{n-1} \in U_{n-1}$ and $Y_{n-1} \in \p_{n-1}$ such that
$$u_{n-1}e^{Y_{n-1}}=e^{X_{n-1}}\dots e^{X_2}e^{Y_1}.$$
Also there are $u_{n-1}' \in U_{n-1}$ and $Y_{n-1}' \in \p_{n-1}$ such that
$$u_{n-1}'e^{Y_{n-1}'}=e^{X_{n-1}'}\dots e^{X_2'}e^{Y_1'}.$$
Then
$$u_nu_{n-1}e^{Ad_{u_{n-1}}^{-1}X_n}e^{Y_{n-1}}=u_n'u_{n-1}'e^{Ad_{{u'}_{n-1}^{-1}}X_n'}e^{Y_{n-1}'}$$
and because of the uniqueness of the splitting theorem for $k=2$ we conclude that
\begin{eqnarray} \label{ig}
u_nu_{n-1}&=&u_n'u_{n-1}' \nonumber \\
Ad_{u_{n-1}^{-1}}X_{n}&=&Ad_{{u'}_{n-1}^{-1}}X_{n}' \\
Y_{n-1}&=&Y_{n-1}'. \nonumber
\end{eqnarray}
Since $u_{n-1}e^{Y_{n-1}}=e^{X_{n-1}}\dots e^{X_2}e^{Y_1}$ and $u'_{n-1}e^{Y_{n-1}'}=e^{X_{n-1}'}\dots e^{X_2'}e^{Y_1'}$
$$u_{n-1}^{-1}e^{X_{n-1}}\dots e^{X_2}e^{Y_1}=e^{Y_{n-1}}=e^{Y_{n-1}'}={u'}_{n-1}^{-1}e^{X_{n-1}'}\dots e^{X_2'}e^{Y_1'}$$
the uniqueness of the splitting theorem for $k = n-1$ implies that $u_{n-1}= u_{n-1}'$, $X_{n-1}=X_{n-1}' $,\dots , $X_2=X_2' $ and $Y_1=Y_1' $.
The equalities in (\ref{ig}) say that $u_n= u_n'$ and $X_n=X_n' $ also hold.

We prove that $\Psi_n$ is bijective based on the fact that $\Phi_n$ is bijective. If $p_n\in G^+_A$ then $p_n=g_ng_n^*$ for some $g_n\in G_n$. Because $\Phi_n$ is surjective there are $u_{n} \in U_{n}$, $X_{n}\in \p_{E_{n}}$,\dots , $X_2 \in \p_{E_2}$ and $Y_1 \in \p_1$ such that $g_n^{*} = u_{n}e^{X_{n}}\dots e^{X_2}e^{Y_1}$. Then $p_n=g_ng_n^*=e^{Y_1}e^{X_2}\dots e^{2X_n}\dots e^{X_2}e^{Y_1}$ and we conclude that $\Psi_n$ is surjective.
To see that $\Psi_n$ is injective let assume that  $e^{Y_1}e^{X_2}\dots e^{2X_n}\dots e^{X_2}e^{Y_1}=e^{Y'_1}e^{X'_2}\dots e^{2X'_n}\dots e^{X'_2}e^{Y'_1}$. If $g_n:=e^{Y_1}e^{X_2}\dots e^{X_n}$ and ${g'}_n:=e^{Y'_1}e^{X'_2}\dots e^{X'_n}$ then $g_ng_n^*={g'}_n{g'}_n^*$ and therefore there is an $u_n\in U_n$ such that $g_nu_n=g'_n$. Then $u_ne^{X_{n}}\dots e^{X_2}e^{Y_1}=e^{X'_{n}}\dots e^{X'_2}e^{Y'_1}$ and we conclude that $(X_n,\dots ,X_2,Y_1)=(X'_n,\dots ,X'_2,Y'_1)$.

We prove that $\Phi_n$ is a diffeomorphism by induction. The CPR splitting states that $\Phi_2$ is a diffeomorphism. Assume that $n>2$ and that $\Phi_{n-1}$ is a diffeomorphism. If $g_n \in G_n$ then $g_n=u_n(g_n)e^{X_n(g_n)}e^{Y_{n-1}(g_n)}$, where $(u_n,X_n,Y_{n-1}): G_n \to U_n \times \p_{E_n} \times \p_{n-1}$ is smooth because the inverse of the CPR splitting is smooth in the case $n=2$. If we denote $f(g_n):=e^{Y{n-1}(g_n)}$ then $f$ is a smooth map and 
$$f(g_n)=u_{n-1}(f(g_n))e^{X_{n-1}(f(g_n))}\dots e^{X_2(f(g_n))}e^{Y_1(f(g_n))}$$
where 
$$(u_{n-1},X_{n-1},\dots ,X_{2},Y_1): G_{n-1}\to U_{n-1}\times \p_{E_{n-1}} \times \dots  \times \p_{E_2} \times \p_1$$
is a smooth map. Since 
$$g_n=u_n(g_n)e^{X_n(g_n)}u_{n-1}(f(g_n))e^{X_{n-1}(f(g_n))}\dots e^{X_2(f(g_n))}e^{Y_1(f(g_n))}=$$
$$u_n(g_n)u_{n-1}(f(g_n))e^{Ad_{u_{n-1}^{-1}(f(g_n))}X_n(g_n)}e^{X_{n-1}(f(g_n))}\dots e^{X_2(f(g_n))}e^{Y_1(f(g_n))}$$ 
we get that $\Phi_n^{-1}: G_n \to U_n \times \p_{E_n} \times \dots  \times \p_{E_2} \times \p_1$ 
$$
g_n \mapsto (u_n(g_n)u_{n-1}(f(g_n)),Ad_{u_{n-1}^{-1}(f(g_n))}X_n(g_n),\dots ,X_2(f(g_n)),Y_1(f(g_n)))
$$ 
is smooth.

We prove next that $\Psi^{-1}=(\overline{X_n},\dots ,\overline{X_2},\overline{Y_1})$ is smooth. Let $g_n\in G_n$, then if $p_n=g_n^*g_n$, 
$$
p_n=e^{(\overline{Y}_1(p_n))}e^{(\overline{X}_2(p_n))}\dots e^{(\overline{X}_{n-1}(p_n))}e^{(2\overline{X}_n(p_n))}e^{(\overline{X}_{n-1}(p_n))}\dots e^{(\overline{X}_2(p_n))}e^{(\overline{Y}_1(p_n))}.
$$ 
Since $g_n=u_n(g_n)e^{X_n(g_n)}\dots e^{X_2(g_n)}e^{Y_1(g_n)}$ where $\Phi^{-1}=(u_n,X_n,\dots ,X_2,Y_1)$, we get 
$$p_n:=g_n^*g_n=e^{Y_1(g_n)}e^{X_2(g_n)}\dots e^{X_{n-1}(g_n)}e^{2X_n(g_n)}e^{X_{n-1}(g_n)}\dots e^{X_2(g_n)}e^{Y_1(g_n)}$$
so that 
$$(\overline{X_n},\dots ,\overline{X_2},\overline{Y_1})=(X_n,\dots ,X_2,Y_1) \circ \pi$$
where $\pi: G_n \to G_n^+, g_n\to g_n^*g_n$. Since $\pi$ is a submersion we conclude that $\Psi^{-1}=(\overline{X}_n,\dots ,\overline{X}_2,\overline{Y}_1)$ is smooth. 

\end{proof}

\begin{rem}
We note that in the context of the previous theorem, if $F_{k,j}:=E_{j+1} \circ \dots  \circ E_{k}$, then $(G_k,G_j;F_{k,j})$ is a reductive structure and $\|{F_{k,j}}\vrule_{\,\p_k}\|=1$.
\end{rem}

\begin{rem}
The splitting theorem of Porta and Recht \cite{portarecht} asserts that if we have a unital inclusion of $C^*$-algebras $B \subseteq A$ and a conditional expectation $E: A\to B$ then the map
$$\Phi: U_A \times \p_E \times \p_B \to G_A$$
$$(u,X,Y) \mapsto ue^Xe^Y$$
is a diffeomorphism, where $\p_E$ are the self-adjoint elements of $Ker E$ and $\p_B$ are the self-adjoint elements of $B$.  

Theorem \ref{descpolext} in the case $n=2$ is a formulation of the CPR splitting (Theorem \ref{cpr}) in the context of reductive structures. The Porta-Recht splitting theorem is a special case of the previous theorem if we consider $(G_A,G_B;E,\sigma)$ derived from the triple $(A,B;E)$ as in Example \ref{condexp} and verify that the conditions of the theorem are satisfied because of what was stated in Example \ref{calg}. The CPR theorem covers the case where the inclusion of algebras and the map $E$ are not unital, as in Example \ref{corners} of reductive structures. It also covers the case where the symmetric space and reductive structure are derived from unitized ideals of operators as in Example \ref{ideal} and Example \ref{idealcond}, see \cite{condelarotonda}.

The CPR theorem in the context of several reductive structures (Theorem \ref{descpolext}) covers for example the case of multiple unital inclusions of $C^*$-algebras and conditional expectations between them 
$$A_1\subseteq A_2\subseteq \dots  \subseteq A_n$$
$$A_1 \xleftarrow{E_2} A_2 \xleftarrow{E_3} \dots  \xleftarrow{E_n} A_n.$$
\end{rem}

\section{Complexifications of homogeneous spaces}
 
Proposition \ref{comp} to Remark \ref{tan} here are extensions of Section 5 of \cite{beltitagale}, from the context of $C^*$-algebras to the context of Finsler symmetric spaces of semi-negative curvature with reductive structures.

\begin{defi}\label{defcom}
Let $X$ be a Banach manifold. A \textit{\textbf{complexification}} of $X$ is a complex Banach manifold Y endowed with an anti-holomorphic involutive diffeomorphism $\sigma$ such that the fixed point submanifold $Y_0 = \{y\in Y:\sigma(y)=y\}$ is a strong deformation retract of $Y$ and is diffeomorphic to $X$. 
\end{defi}

\begin{ejem}
Let $M=G/U=Sym(G,\sigma,\|\cdot \|)$ be a simply connected Finsler symmetric space of semi-negative curvature. Theorem \ref{polar} guaranties that $U$ is a strong deformation retract of $G$. If $G$ is a complex analytic group and $\sigma$ is anti-holomorphic, then $G$ is a complexification of $U$.
In the context of $C^*$-algebras the group of invertible elements $G$ is a complexification of the group of unitary elements $U$ with $\sigma=(\cdot)^{-1*}$. Note that $U$ is not a complex analytic manifold.
\end{ejem}

\begin{defi}
Let $(G_A,\sigma)$ be a symmetric Banach-Lie group with involutive subgroup $G_B$. We define $\sigma_G: G_A/G_B \to G_A/G_B$, $uG_B \mapsto \sigma(u)G_B$ and $\lambda: U_A/U_B \hookrightarrow G_A/G_B$, $uU_B \mapsto uG_B$.
\end{defi}

We now give a criteria which implies that $U_A/U_B$ is diffeomorphic to the fixed point set of the involution $\sigma_G$.

\begin{prop}\label{pos}
If $M_A=G_A/U_A=Sym(G_A,\sigma,\|\cdot \|)$ is a Finsler symmetric space of semi-negative curvature, $G_B$ is an involutive subgroup of $G_A$, and  $\Gamma \subseteq \p_B$, then $G_A^+ \cap G_B = G_B^+$.
\end{prop}
\begin{proof}
Since $G_B^+ \subseteq  G_A^+ \cap G_B$ always holds, it is enough to prove that $G_A^+ \cap G_B \subseteq G_B^+$. By Cor. \ref{exppos} $G_A^+ = e^{\p}$ and $G_B^+ = e^{\p_B}$.
If $g \in G_A^+ \cap G_B$ then there is an $X\in \p$ such that $g=e^X$. Since $G_B$ is an involutive subgroup $G_B/U_B$ has semi-negative curvature and using the polar decomposition of Th. \ref{polar} in $G_B$ guaranties the existence of $u\in U_B$ and $Y\in \p_B$ such that $g=ue^Y$. Then, Theorem \ref{polar} applied to $G_A$ tells us that for certain $Z\in \Gamma$, $u=e^Z$ and $Y=X-Z$. Since $\Gamma \subseteq \g_B$ we conclude that $X \in \g_B$ and therefore $g \in G_B^+$.
\end{proof}

\begin{prop}\label{comp}
If $G_B^+ = G_A^+ \cap G_B$, then $\lambda(U_A/U_B) = \{s \in G_A/G_B : \sigma_G (s) = s\}$.
\end{prop}
\begin{proof}
The inclusion $\subseteq$ is obvious. Given $s=uG_B$ such that $\sigma_G(s)=s$, $u^{-1}\sigma(u) \in G_B$. Since $u^{-1}\sigma(u) \in G_A^+$ the hypothesis $G_B^+ = G_A^+ \cap G_B$ implies that $u^{-1}\sigma(u) \in G_B^+$, and therefore there exists $w\in G_B$ such that $u^{-1}\sigma(u)=ww^*$. Then $uw=\sigma(u)w^{*-1}=\sigma(u)\sigma(w)=\sigma(uw)$, so that $uw\in U_A$ and $s=uG_B=uwG_B=\lambda(uwU_B)$.
\end{proof}

We give a geometric description of the complexification $G_A/G_B$ of $U_A/U_B$ in the context of reductive structures. This can be seen as an infinite dimensional version of Mostow fibration, see \cite{mostow,mostow2} and Section 3 of \cite{bielawski}.

\begin{rem}
Since the actions $Ad:U_B \to {\cal B}(\p_E)$ and $Ad:U_B \to {\cal B}(\u_E)$ are well defined we get the homogeneous vector bundles $U_A \times_{U_B} \p_E \to U_A/U_B$ and $U_A \times_{U_B} \u_E \to U_A/U_B$, $[(u,X)] \mapsto uU_B$, where the actions of $U_B$ on $U_A \times_{U_B} \p_E$ and $U_A \times_{U_B} \u_E$ are given by $v\cdot(u,X) = (uv^{-1},Ad_{v}X)$.
\end{rem}

\begin{teo}\label{com}
Let $M_A=G_A/U_A=Sym(G_A,\sigma,\|\cdot \|)$ be a simply connected Finsler symmetric space of semi-negative curvature  and $(G_A,G_B;E,\sigma)$ a reductive structure with involution such that $\|E\vrule_{\,\p}\|=1$. Consider  $\Psi_0^E: U_A \times \p_E \to G_A$, $(u,X)\mapsto ue^X$ and $\kappa: (u,X)\mapsto [(u,X)]$ the quotient map. Then there is a unique real analytic, $U_A$-equivariant diffeomorphism $\Psi^E: U_A \times_{U_B} \p_E \to G_A/G_B$ such that the diagram
$$
\xy
\xymatrix{
U_A \times \p_E \ar[d]_{\kappa} \ar[r]^{\Psi_0^E} &  G_A \ar[d]^{q}
\\
U_A \times_{U_B} \p_E \ar[r]^{\Psi^E} & G_A/G_B
}\endxy
$$
commutes.

Therefore the homogeneous space $G_A/G_B$ has the structure of an $U_A$-equivariant fiber bundle over $U_A/U_B$ with the projection given by the composition
$$
\xy
\xymatrix{
G_A/G_B  \ar[r]^{(\Psi^E)^{-1}} &  U_A \times_{U_B} \p_E \ar[r]^{\Xi} & U_A/U_B
}\endxy
$$
$$ue^XG_B \mapsto [(u,X)] \mapsto uU_B \qquad \textrm{for $u \in U_A$ and $X \in \p_E$}$$
and typical fiber $\p_E$.
\end{teo}

\begin{proof}
To prove that $\Psi^E$ is well defined we show that for $u\in U_A$, $v\in U_B$ and $X\in \p_E$
\begin{eqnarray}
q(\Psi_0^E(u,X)) &=ue^XG_B=uv^{-1}e^{Ad_vX}vG_B=uv^{-1}e^{Ad_vX}G_B \nonumber \\
&=q(\Psi_0^E(uv^{-1},Ad_vX))=q(\Psi_0^E(v\cdot(u,X))) \nonumber
\end{eqnarray}

The uniqueness of $\Psi^E$ is a consequence of the surjectivity of $\kappa$.

Theorem \ref{descpolext} for the case $n=2$ implies the existence of a diffeomorphism
$$\Phi: U_A \times \p_E \times \p_B \to G_A$$
$$(u,X,Y) \mapsto ue^Xe^Y.$$
If $gG_B\in G_A/G_B$ there is $(u,X,Y) \in U_A \times \p_E \times \p_B$ such that $g=ue^Xe^Y$ and we get $gG_B=ue^Xe^YG_B=ue^XG_B$, proving the surjectivity of $\Phi$.

To see that $\Psi^E$ is also injective assume that $u_1e^{X_1}G_B=u_2e^{X_2}G_B$. Then there is a $b\in G_B$ such that $u_1e^{X_1}b=u_2e^{X_2}$. Since $G_B$ is an involutive connected subgroup of $G_A$ and $G_A/U_A$ has semi-negative curvature, Proposition \ref{subgrupoinv} states that $G_B/U_B$ has also semi-negative curvature and we can apply the polar decomposition (Proposition \ref{polar}) in $G_B$: there are unique $v\in U_B$ and $Y\in \p_B$ such that $b=ve^Y$. Then
$$(u_1v)e^{Ad_{v^{-1}}X_1}e^Y=u_1e^{X_1}ve^Y=u_1e^{X_1}b=u_2e^{X_2}$$
and applying $(\Phi)^{-1}$ to this equality we get $(u_1v,Ad_{v^{-1}}X_1,Y)=(u_2,X_2,0)$, which implies that $v^{-1} \cdot (u_1,X_1)=(u_2,X_2)$.

Finally, we prove that $\Psi^E$ is an analytic diffeomorphism. Since $\kappa$ is a submersion and $\Psi^E \circ \kappa$ ($=q \circ \Psi_0^E$) is a real analytic map $\Psi^E$ is real analytic. Since the map $\Phi^{-1}: g\mapsto (u(g),X(g),Y(g))$ is analytic, the map $\sigma: g \mapsto [(u(g),X(g))]$, $G_A \to U_A \times_{U_B} \p_E$ is also analytic. Since $q$ is a submersion and $\sigma = (\Psi^E)^{-1} \circ q$ we see that $(\Psi^E)^{-1}$ is analytic.
\end{proof}

\begin{coro}
If we analyse the diagram of the previous theorem in the tangent spaces using the following identifications $T_{(1,0)}(U_A \times \p_E) \simeq \u_A \times \p_E$, $T_{[(1,0)]}(U_A \times_{U_B} \p_E) \simeq \u_E \times \p_E$ and $T_o(G_A/G_B) \simeq KerE$ then
$$(\Phi_0^E)_{*(1,0)}: \u_A \times \p_E \to \g_A, \qquad (Y,Z) \mapsto Y+Z$$
$$\kappa_{*(1,0)}: \u_A \times \p_E \to \u_E \times \p_E,\qquad (Y,Z) \mapsto ((1-E)Y,Z)$$
$$q_{*1}: \g_A \mapsto KerE,\qquad W \mapsto (1-E)W$$
and therefore
$$(\Phi^E)_{*[(1,0)]}: \u_E \times \p_E \to KerE,\qquad ((1-E)Y,Z) \mapsto (1-E)(Y+Z)=(1-E)Y + Z.$$
We conclude that
$$(\Phi^E)_{*[(1,0)]}: \u_E \times \p_E \to KerE,\qquad (X,Z) \mapsto X + Z$$
is an isomorphism.
\end{coro}

\begin{coro}
If we assume the conditions of the previous theorem, the fixed point set of of the involution $\sigma_G$ on $G_A/G_B \simeq U_A \times_{U_B} \p_E$ is diffeomorphic to $U_A/U_B$ and $U_A/U_B$ is a strong deformation retract of $G_A/G_B$. If $G_A$ is a complex analytic group and $\sigma$ is anti-holomorphic then $G_A/G_B$ is a complexification of $U_A/U_B$.

If we define $\tau_G : U_A \times_{U_B} \p_E \to U_A \times_{U_B} \p_E$, $[(u,X)] \mapsto [(u,-X)]$, then the following diagram
$$
\xy
\xymatrix{
U_A \times_{U_B} \p_E \ar[d]_{\Psi^E} \ar[r]^{\tau_G} &  U_A \times_{U_B} \p_E \ar[d]^{\Psi^E}
\\
G_A/G_B \ar[r]^{\sigma_G} & G_A/G_B
}\endxy
$$
commutes.
\end{coro}
\begin{proof}
Note that $\Gamma = \{0\}$ so that Prop. \ref{pos} implies  $G_B^+=G_B \cap G_A^+$ and Prop. \ref{comp} states that $U_A/U_B$ is diffeomorphic to the set of fixed points on $\sigma_G$.

Alternatively, the diagram tells us that the set of fixed points of the involution $\sigma_G$ is $\Psi^E(\{[(u,X)] \in U_A \times_{U_B} \p_E: \tau_G([(u,X)])=[(u,X)]\})=\Psi^E(\{[(u,0)]: u \in U_A \})=\{uG_B: u \in U_A\}=\lambda(U_A/U_B)$.

If we define $F: (U_A \times_{U_B} \p_E) \times [0,1] \to U_A \times_{U_B} \p_E, ([(u,X)],t) \mapsto [(u,tX)]$ we see that $\{[(u,0)]: u \in U_A \}$ is a strong deformation retract of $U_A \times_{U_B} \p_E$ and $\{[(u,0)]: u \in U_A \}$ is diffeomorphic to $U_A/U_B$.

If $\sigma$ is anti-holomorphic then $\sigma_G$ is anti-holomorphic (see \cite{neeb2}) and it follows from Definition \ref{defcom} that $G_A/G_B$ is a complexification of $U_A/U_B$.
\end{proof}

\begin{teo}\label{tang}
If we assume that the conditions of Theorem \ref{com} are satisfied then there is a $U_A$-equivariant diffeomorphic vector bundle map from the associated vector bundle $U_A \times_{U_B} \u_E \to U_A/U_B$ to the tangent bundle $T(U_A/U_B) \to U_A/U_B$ given by $\alpha^E:U_A \times_{U_B} \u_E \to T(U_A/U_B)$, $[(u,X)] \mapsto (\mu_u)_{*o}q_{*1}X$, where the action of $U_A$ on $T(U_A/U_B)$ is given by $u \cdot - = (\mu_u)_*-$ for every $u \in U_A$.
\end{teo}
\begin{proof}
Let $\alpha : U_A \times U_A/U_B \to U_A/U_B$ be given by $(u,vU_B) \mapsto uvU_B$, then $\partial_2\alpha : U_A \times T(U_A/U_B) \to T(U_A/U_B)$, $(u,V) \mapsto (\mu_u)_*V$. Since $E \circ \sigma_{*1} = \sigma_{*1} \circ E$ $E(\u) \subseteq \u$, and since $E(\g_A) = \g_B$ we get the decomposition $\u = \u_B \oplus \u_E$. Then $\u_E \simeq T_o(U_A/U_B)$, $X \mapsto q_{*1}X$ and restricting $\partial_2\alpha$ to $U_A \times T_o(U_A/U_B)$ we get a map $\alpha_0^E: U_A \times \u_E \to T(U_A/U_B)$, $(u,X) \mapsto (\mu_u)_{*o}q_{*1}X$.

We assert that there is a unique $U_A$-equivariant diffeomorphism $\alpha^E:U_A \times_{U_B} \u_E \to T(U_A/U_B)$ such that $\alpha^E \circ \kappa = \alpha_0^E$, where $\kappa$ is the quotient map $(u,X) \mapsto [(u,X)]$.

To prove that $\alpha^E$ is well defined we see that for every $u \in U_A$, $v \in U_B$ and $X \in \u_E$
\begin{eqnarray}
\alpha_0^E(v \cdot(u,X))&=& \alpha_0^E(uv^{-1},Ad_vX)=(\mu_{uv^{-1}})_{*o}q_{*1}Ad_vX \nonumber \\
&=&(\mu_{uv^{-1}})_{*o}q_{*1}(I_v)_{*1}X=(\mu_{uv^{-1}}qI_v)_{*1}X \nonumber \\
&=&(\mu_u\mu_{v^{-1}}qL_vR_{v^{-1}})_{*1}X=(\mu_uqL_{v^{-1}}L_vR_{v^{-1}})_{*1}X \nonumber \\
&=&(\mu_uqR_{v^{-1}})_{*1}X=(\mu_uq)_{*1}=(\mu_u)_{*o}q_{*1}X=\alpha_0^E(u,X) \nonumber
\end{eqnarray}
The uniqueness of $\alpha^E$ is a consequence of the surjectivity of $\kappa$.
$\alpha^E$ is surjective because $(\mu_u)_{*o}:T_o(U_A/U_B) \to T_{q(u)}(U_A/U_B)$ is bijective for every $u \in U_A$.
To see that $\alpha^E$ is injective assume that $(\mu_{u_1})_{*o}q_{*1}X_1=(\mu_{u_2})_{*o}q_{*1}X_2$. Then
$q(u_1)=q(u_2)$ and therefore there is a $v \in U_B$ such that $u_1v=u_2$. Then
\begin{eqnarray}
(\mu_{u_1})_{*o}q_{*1}X_1 &=&(\mu_{u_2})_{*o}q_{*1}X_2=(\mu_{u_1v}q)_{*1}X_2=(\mu_{u_1}\mu_vq)_{*1}X_2 \nonumber \\
&=&(\mu_{u_1}\mu_vqR_{v^{-1}})_{*1}X_2=(\mu_{u_1}qL_vR_{v^{-1}})_{*1}X_2 \nonumber \\
&=&(\mu_{u_1}qI_v)_{*1}X_2=(\mu_{u_1})_{*o}q_{*1}Ad_vX_2 \nonumber
\end{eqnarray}
so that $Ad_vX_2=X_1$ and we conclude that $v \cdot (u_2,X_2) = (u_1,X_1)$.
\end{proof}

\begin{lem}
If $\sigma$ is a anti-holomorphic involutive automorphism of a complex Banach-Lie group $G_A$ then $i\u=\p$.
\end{lem}
\begin{proof}
If $X\in \u$, $\sigma_{*1}X=X$ and $\sigma_{*1}(iX)=-i\sigma_{*1}X=-iX$ so that $iX\in \p$. The other inclusion is proved in a similar way.
\end{proof}

\begin{ejem} 
If $G_A$ is the group of invertible elements of a $C^*$-algebra $A$ then the previous lemma applies and we get $\p=A_s$ the set of self-adjoint elements of $A$ and $\u=i\p=iA_s=A_{as}$ the set of skew-adjoint elements of $A$.
\end{ejem}

\begin{rem}\label{tan}
Assume the conditions of Theorem \ref{com} are satisfied and that $G_A$ is a complex analytic group, $\u =i\p$, and $E$ is $\mathbb{C}$-linear. Since $Ad_g(iX)=iAd(X)$ for every $g\in G_A$ and $X\in \g_A$ we conclude that $\Theta: U_A \times_{U_B} \p_E \to U_A \times_{U_B} \u_E$, given by $[(u,X)] \mapsto [(u,iX)]$ is well defined. Theorem \ref{com} and Theorem \ref{tang} imply that the composition
$$G_A/G_B \xrightarrow{(\Psi^E)^{-1}} U_A \times_{U_B} \p_E \xrightarrow{\Theta} U_A \times_{U_B} \u_E \xrightarrow{\alpha^E} T(U_A/U_B)$$
is a $U_A$-equivariant diffeomorphism between the complexification $G_A/G_B$ and the tangent bundle $T(U_A/U_B)$ of the homogeneous space $U_A/U_B$. Under the above identification the involution $\sigma_G$ is the map $V\mapsto -V$, $T(U_A/U_B)\to T(U_A/U_B)$.         
\end{rem}

The isomorphism in the last remark gives the tangent bundle of $U_A/U_B$ a complex manifold structure which depends on the map $E$.

\medskip

The following proposition shows that the diffeomorphism between $G_A/G_B$ and $T(U_A/U_B)$ respects the natural morphisms that can be defined between homogeneous spaces of the form $G_A/G_B$ and tangent bundles of homogeneous spaces given by $T(U_A/U_B)$.

\begin{prop}
Let $(G_A,G_B;E;\sigma)$ and $(\tilde{G_A},\tilde{G_B};\tilde{E};\tilde{\sigma})$ be reductive structures with involution that satisfy the conditions of the previous remark and let $\alpha: G_A \to \tilde{G_A}$ be a holomorphic morphism of reductive structures with involution. If we define $\alpha_G:G_A/G_B \to G_A/G_B$, $gG_B \mapsto \alpha(g)\tilde{G_B}$ and $\alpha_U:U_A/U_B \to U_A/U_B$, $uU_B \mapsto \alpha(u)\tilde{U_B}$ then the diagram
$$
\xy
\xymatrix{
\ G_A/G_B \ar[d]_{\alpha_G} & U_A \times_{U_B} \u_E \ar[l]_{\sim} \ar[r]^{\sim} & T(U_A/U_B) \ar[d]^{\alpha_{U*}}
\\
\tilde{G_A}/\tilde{G_B} & \tilde{U_A} \times_{\tilde{U_B}} \tilde{\u_E} \ar[l]_{\sim} \ar[r]^{\sim} & T(\tilde{U_A}/\tilde{U_B})
}\endxy
$$
commutes, where the horizontal arrows correspond to the morphisms of Rem. \ref{tan}.
\end{prop}
\begin{proof}
Since $\alpha \circ \sigma = \tilde{\sigma} \circ \alpha$, $\alpha(U_B)\subseteq \tilde{U_B}$ and $\alpha_U$ is well defined.
Since $\alpha_{*1} \circ \sigma_{*1} = \tilde{\sigma}_{*1} \circ \alpha_{*1}$, $\alpha_{*1}(\u) \subseteq \tilde{\u}$. Also $E \circ \alpha_{*1} = \alpha_{*1} \circ E$ implies $\alpha_{*1}(KerE) \subseteq Ker\tilde{E}$ so that $\alpha_{*1}(\u_E) \subseteq \tilde{\u_E}$.
Given $u \in U_A$ and $X\in \u_E$, $\alpha(u)\in \tilde{U_A}$ and $\alpha_{*1}X \in \tilde{u_E}$  and we have the following diagram
$$
\xy
\xymatrix{
\ ue^{iX}G_B  \ar@{|->}[d]_{\alpha_G} & [(u,X)] \ar@{|->}[l]^{} \ar@{|->}[r]^{} & (\mu_u)_{*o}q_{*1}X \ar@{|->}[d]^{\alpha_{U*}}
\\
\ \alpha(u)e^{i\alpha_{*1}(X)}\tilde{G_B} & [(\alpha(u),\alpha_{*1}(X))]  \ar@{|->}[l]^{} \ar@{|->}[r]^{} & (\tilde{\mu}_{\alpha(u)})_{*o}\tilde{q}_{*1}\alpha_{*1}X
}\endxy
$$
It is enough to verify that the values in the vertical arrows correspond to the stated morphisms.
$\alpha_G(ue^{iX}G_B)=\alpha(u)e^{\alpha_{*1}(iX)}\tilde{G_B}=\alpha(u)e^{i\alpha_{*1}(X)}\tilde{G_B}$ since $\alpha_{*1}(iX)=i\alpha_{*1}(X)$ because $\alpha$ is holomorphic.
Since $\alpha_U \circ \mu_{u} = \tilde{\mu}_{\alpha(u)} \circ \alpha_U$ and $\tilde{q} \circ \alpha = \alpha_U \circ q$ we get $\alpha_{U*q(u)} (\mu_u)_{*o} q_{*1}X = (\tilde{\mu}_{\alpha(u)})_{*o}  \alpha_{U*o} q_{*1}X = (\tilde{\mu}_{\alpha(u)})_{*o}\tilde{q}_{*1}\alpha_{*1}X$.
\end{proof}

There are two basic examples of homogeneous spaces $U_A/U_B$ in the infinite dimensional context, the flag manifolds and the Stiefel manifolds. Coadjoint orbits are examples of flag manifolds.

\begin{ejem}Flag manifolds

Let $\h$ be a Hilbert space and let $p_i$, $i=1,\dots ,n$ be pairwise orthogonal projections in $\b(\h)$ each with range $\h_i$ such that $\sum_{i=1}^{n}p_i =1$. If we consider the action of the unitary group $U_A$ of $\b(\h)$ on the set of $n$-tuples of pairwise orthogonal projections with sum $1$ given by $u \cdot (q_1,\dots ,q_n)=(uq_1u^*,\dots ,uq_nu^*)$ then the orbit of $(p_1,\dots ,p_n)$ can be considered as an infinite dimensional version of a flag manifold. This orbit is isomorphic to $U_A/U_B$ where
$$
U_B=
\left\{ \begin{pmatrix}
u_1    & 0      & \hdots & 0      \\
0      & u_2    & \hdots &  0     \\
\vdots & \vdots & \ddots & \vdots \\
0      &  0     &  \hdots   &  u_n
\end{pmatrix} :u_i \mbox{ unitary in } \b(\h_i) \mbox{ for } i=1,\dots ,n \right\};
$$
and we write the operators in $\b(\h)=\b(\h_1 \oplus \dots  \oplus \h_n)$ as $n\times n$-matrices with corresponding operator entries. If we consider the group $G_A$ of invertible operators in $\b(\h)$ with the usual involution $\sigma$, the involutive subgroup
$$
G_B=
\left\{ \begin{pmatrix}
g_1    & 0      & \hdots & 0      \\
0      & g_2    & \hdots &  0     \\
\vdots & \vdots & \ddots & \vdots \\
0      &  0     &  \hdots   &  g_n
\end{pmatrix} :g_i \mbox{ invertible in } \b(\h_i) \mbox{ for } i=1,\dots ,n \right\};
$$
and the conditional expectation $E:\g_A \to \g_B$, $X\mapsto \sum_{i=1}^{n}p_iXp_i$ then we are in the context of Example \ref{condexp} and Th. \ref{com}, Th. \ref{tang} and Rem. \ref{tan} give a geometric description of the complexification of the flag manifold.
\end{ejem}

Other examples of flag manifolds in the infinite dimensional context are coadjoint orbits in operator ideals, which now can be described geometrically.

\begin{ejem}Coadjoint orbits

In the setting of Example \ref{idealcond} let $1\le p <\infty$ and $q$ such that $1/p+1/q=1$. The Banach-Lie algebra of the Banach-Lie group $G_{A,p}$ is $\g_{A,p}=A_p$, the ideal of p-Schatten operators ($A_\infty$ is the ideal of compact operators). The Banach-Lie algebra of the real Banach-Lie group $U_{A,p}$ is $\u_{A,p}$, the skew-adjoint p-Schatten operators. The trace provides strong duality pairings $\g_{A,p}^*\simeq \g_{A,q}$ and $\u_{A,p}^*\simeq \u_{A,q}$. 

We denote by $Ad^*:G_{A,p}\mapsto \b(\g_{A,p})$, $Ad^*_g(X)=(Ad_{g^{-1}})^*(X)=gXg^{-1}$ for $g\in G_{A,p}$ and $X\in \g_{A,p}^*\simeq \g_{A,q}$, the coadjoint action of $G_{A,p}$ and by $Ad^*:U_{A,p}\mapsto \b(\u_{A,p})$, $Ad^*_u(X)=(Ad_{u^{-1}})^*(X)=uXu^{-1}$ for $u\in U_{A,p}$ and $X\in \u_{A,p}^*\simeq \u_{A,q}$, the coadjoint action of $U_{A,p}$. 

For a fixed $X \in \u_{A,q}\subseteq \g_{A,q}$ let $\o_G(X)=\{Ad^*_g(X):g\in G_{A,p}\}$ be the coadjoint orbit of $X$ under the action of $G_{A,p}$ and $\o_U(X)=\{Ad^*_u(X):g\in U_{A,p}\}$ be the coadjoint orbit of $X$ under the action of $U_{A,p}$. Since $X$ is a compact skew-adjoint operator it is diagonalizable, i.e. there is a finite or countable sequence of pairwise orthogonal projections $(p_i)_{i=1}^N$ with $N\in \mathbb{N}\cup\{\infty\}$ such that $\sum_{i=1}^{N}p_i=1$ and $X= \sum_{i=1}^{N}\lambda_ip_i$, where $\lambda_i\neq \lambda_j$ for $i\neq j$ and $(\lambda_i)_{i=1}^N\subseteq i\mathbb{R}$. The map $E:Y\mapsto \sum_{i=1}^{N}p_iYp_i$ is a conditional expectation from $A$ onto the $C^*$-subalgebra $B=\{Y\in A: p_iY=Yp_i \mbox{ for all } i\geq 1\}$. This conditional expectation sends trace-class operators to trace-class operators and preserves the trace, so the conditions on $E$ in Example \ref{idealcond} are satisfied.  The coadjoint isotropy group of $X$ for the action of $G_{A,p}$ is $\{g\in G_{A,p}:gXg^{-1}=X\}=G_{B,p}$ and the coadjoint isotropy group of $X$ for the action of $U_{A,p}$ is $\{u\in U_{A,p}:uXu^{-1}=X\}=U_{B,p}$. This follows from the fact that an operator commutes with a diagonalizable operator if and only if it leaves all the eigenspaces of the diagonalizable operator invariant. Thus, making the identifications $\o_G(X)\simeq G_{A,p}/G_{B,p}$ and $\o_U(X)\simeq U_{A,p}/U_{B,p}$, Th. \ref{com}, Th. \ref{tang} and Rem. \ref{tan} give a geometric description of the complexification of the flag manifold; there is a $U_{A,p}$-equivariant diffeomorphic fiber bundle map between $\o_G(X)$ and $T(\o_U(X))$ covering the identity map of $\o_U(X)$. 

For further reading on the coadjoint orbits in the infinite dimensional setting, see Section 7 in \cite{oratiu}.
\end{ejem}

Likewise, it is now possible to give a geometric description of the complexification of the Stiefel manifolds.

\begin{ejem}Stiefel manifolds

Let $\h$ be a Hilbert space and let $p_i$, $i=1,2$ be pairwise orthogonal projections in $\b(\h)$ each with range $\h_i$ such that $p_1+p_2=1$. If we consider the action of the unitary group $U_A$ of $\b(\h)$ on the set of partial isometries given by by $u \cdot v=uv$ then the orbit of $p_1$ can be considered as an infinite dimensional version of a Stiefel manifold. This orbit is isomorphic to $U_A/U_B$ where
$$
U_B=
\left\{ \left(\begin{array}{cc}
1 & 0 \\
\\
0 & u
\end{array}\right):u\mbox{ is unitary in }\b(\h_2) \right\}.
$$
and we write the operator in $\b(\h)=\b(\h_1 \oplus \h_2)$ as $2\times 2$-matrices with corresponding operator entries. If we consider the group $G_A$ of invertible operators in $\b(\h)$ with the usual involution $\sigma$, the involutive subgroup
$$
G_B=
\left\{ \left(\begin{array}{cc}
1 & 0 \\
\\
0 & g
\end{array}\right):g\mbox{ is invertible in }\b(\h_2) \right\}.
$$
and the map $E:\g_A\to \g_B$, $X\mapsto (1-p)X(1-p)$ then we are in the context of Example \ref{corners} and Th. \ref{com}, Th. \ref{tang} and Rem. \ref{tan} give a geometric description of the complexification of the Stiefel manifold.
\end{ejem}

\begin{rem}
The case of the flag manifold with two projections is the infinite dimensional Grassmannian. The case of the Grassmannian where the decomposition of $\h$ is $\h= \mathbb{C}\eta \oplus (\mathbb{C}\eta)^{\perp}$ for a non-zero vector $\eta \in \h$ is the projective space $\mathbb{P}(\h)$.

The special case of the Stiefel manifold where the decomposition of $\h$ is $\h= \mathbb{C}\eta \oplus (\mathbb{C}\eta)^{\perp}$ for a non-zero vector $\eta \in \h$ is the unit sphere in the Hilbert space $\h$.
\end{rem}

\subsection*{Acknowledgements}
I would like to thank Gabriel Larotonda for many useful suggestions and comments.

\end{document}